\documentclass[11pt]{article}
\usepackage{amsfonts,amscd,amssymb, amsmath}
\newtheorem{definition}{Definition}[section]

\newtheorem{proposition}[definition]{Proposition}

\newtheorem{remark}[definition]{Remark}

\newtheorem{theorem}[definition]{Theorem}
\newtheorem{example}[definition]{Example}

\newcommand{\oot}{\overline{\otimes}}
\def\rawo\lonra{\longrightarrow}

\def\ot{\otimes}

\newcommand{\selabel}[1]{\label{se:#1}}

\allowdisplaybreaks[4]

\newenvironment{proof}{{\it Proof.}}{\hfill $ \square $ \vskip 4mm}

\begin{document}
\title{Iterated crossed products
\thanks{This work was supported by a grant of the Romanian National 
Authority for Scientific Research, CNCS-UEFISCDI, 
project number PN-II-ID-PCE-2011-3-0635,  
contract nr. 253/5.10.2011.}}
\author {Florin Panaite\\
Institute of Mathematics of the 
Romanian Academy\\ 
PO-Box 1-764, RO-014700 Bucharest, Romania\\
e-mail: Florin.Panaite@imar.ro
}
\date{}
\maketitle
\begin{center}
Dedicated to Professor Constantin N\u{a}st\u{a}sescu on the 
occasion of his $70^{th}$ birthday
\end{center}

\begin{abstract}
We define a ''mirror version'' of Brzezi\'{n}ski's crossed product and we prove that, 
under certain circumstances, a Brzezi\'{n}ski crossed product $D\otimes _{R, \sigma }V$ 
and a mirror version $W\overline{\otimes }_{P, \nu }D$ may be iterated, obtaining an 
algebra structure on $W\otimes D\otimes V$. Particular cases of this construction are the 
iterated twisted tensor product of algebras and the quasi-Hopf two-sided smash product. 
\end{abstract}
\section*{Introduction}
${\;\;\;\;}$
If $A$ and $B$ are associative unital algebras and $R:B\otimes A\rightarrow 
A\otimes B$ is a linear map satisfying certain axioms (such an $R$ is called a twisting map) 
then $A\otimes B$ becomes an associative unital algebra with a multiplication defined 
in terms of $R$ and the multiplications of $A$ and $B$; this algebra structure is denoted 
by $A\otimes _RB$ and called the twisted tensor product of $A$ and $B$ afforded by $R$ 
(cf. \cite{Cap}, \cite{VanDaele}).  There exist many concrete examples of twisted tensor 
products, for instance the Hopf smash product and other kinds of products arising 
in Hopf algebra theory. 

In \cite{jlpvo}, two types of general results have been proved for twisted tensor products. 
One was called ''invariance under twisting'' (since it was directly inspired by the invariance 
under twisting of the Hopf smash product). The other one is the fact that twisted 
tensor products may be iterated. More precisely, given three twisted tensor products 
$A\otimes _{R_1}B$, $B\otimes _{R_2}C$ and $A\otimes _{R_3}C$, it has been proved that 
a sufficient condition for being able to define certain twisting maps 
$T_1:C\otimes (A\otimes _{R_1}B)\rightarrow (A\otimes _{R_1}B)\otimes C$  and 
$T_2:(B\otimes _{R_2}C)\otimes A\rightarrow A\otimes (B\otimes _{R_2}C)$  associated 
to $R_1$, $R_2$, $R_3$ and ensuring that the algebras $A\otimes _{T_2}(B\otimes _{R_2}C)$ 
and $(A\otimes _{R_1}B)\otimes _{T_1}C$ are equal (this algebra is called the iterated 
twisted tensor product), can be given in terms of the maps $R_1$, $R_2$, $R_3$, 
namely, they have to satisfy the braid relation $(id_A\otimes R_2)\circ 
(R_3\otimes id_B)\circ (id_C\otimes R_1)=(R_1\otimes id_C)\circ (id_B\otimes R_3)\circ 
(R_2\otimes id_A)$. 

On the other hand, there exist important examples of products of algebras that are not 
twisted tensor products, a prominent example being  the classical Hopf crossed 
product. In \cite{brz}, Brzezi\'{n}ski introduced a very general construction, 
called crossed product, containing as particular cases  twisted tensor products of algebras 
as well as classical Hopf crossed products. 
Given an associative unital algebra $A$, a vector space $V$ endowed with a  
distinguished element $1_V$  and two linear maps $\sigma :V\otimes V
\rightarrow A\otimes V$ and $R:V\otimes A\rightarrow A\otimes V$ 
satisfying certain conditions, Brzezi\'{n}ski's crossed product is a certain 
associative unital algebra structure on $A\otimes V$, denoted in what follows by 
$A\otimes _{R, \sigma }V$.  Another example of a crossed product is the smash product 
$H\# B$ between a quasi-bialgebra $H$ and a right $H$-module algebra $B$, cf. \cite{bpvo} 
(since in general $B$ is not associative, $H\# B$ in general is not a twisted tensor product 
of algebras).  

In \cite{panaite} we have proved a result of the type invariance under twisting 
for crossed products, that arose as a common generalization of the invariance under twisting 
for twisted tensor products of algebras and the invariance under twisting for quasi-Hopf 
smash products.   

In this paper we will study iterated crossed products. Our motivating example was the so-called 
quasi-Hopf two-sided smash product $A\# H\# B$ from \cite{bpvo}, where $H$ is a 
quasi-bialgebra and $A$ (respectively $B$) is a left (respectively right) $H$-module 
algebra. We wanted to express $A\# H\# B$ as some sort of iterated product between 
the two smash products $A\# H$ and $H\# B$. As we have seen before, 
$H\# B$ is a Brzezi\'{n}ski crossed product. We needed to express also $A\# H$ as some 
sort of crossed product. It turns out that there exists a ''mirror version'' of 
Brzezi\'{n}ski's crossed product, which we denote by $W\overline{\otimes}_{P, \nu }D$ 
(where $D$ is an associative algebra, $W$ is a vector space and $P$, $\nu $ are certain 
maps), and $A\# H$ is an example of such a crossed product.  Our result, that contains as 
particular cases both the above $A\# H\# B$ and the result about iterated twisted tensor 
products of algebras, may be formulated as follows: if $W\overline{\otimes}_{P, \nu }D$ 
and $D\otimes _{R, \sigma }V$ are two crossed products and $Q:V\otimes W\rightarrow 
W\otimes D\otimes V$ is a linear map satisfying certain conditions, then one can define 
certain maps $\overline{\sigma }$, $\overline{R}$, $\overline{\nu }$, $\overline{P}$ such 
that we have the crossed products $(W\overline{\otimes}_{P, \nu }D)\otimes _{\overline{R}, 
\overline{\sigma }}V$ and $W\overline{\otimes }_{\overline{P}, \overline{\nu }}
(D\otimes _{R, \sigma }V)$ that are moreover equal as algebras (this algebra structure is called 
the iterated crossed product). Moreover, this result admits a certain converse. Finally, we 
prove that a certain construction introduced  in \cite{sommer} is an 
example of an iterated crossed product. 
\section{Preliminaries}\selabel{1}
${\;\;\;\;}$
We work over a commutative field $k$. All algebras, linear spaces
etc. will be over $k$; unadorned $\ot $ means $\ot_k$. By ''algebra'' we 
always mean an associative unital algebra. The multiplication 
of an algebra $A$ is denoted by $\mu _A$ or simply $\mu $ when 
there is no danger of confusion, and we usually denote 
$\mu _A(a\ot a')=aa'$ for all $a, a'\in A$. 

We recall from \cite{Cap}, \cite{VanDaele} that, given two algebras $A$, $B$ 
and a $k$-linear map $R:B\ot A\rightarrow A\ot B$, with notation 
$R(b\ot a)=a_R\ot b_R$, for $a\in A$, $b\in B$, satisfying the conditions 
$a_R\otimes 1_R=a\otimes 1$, $1_R\otimes b_R=1\otimes b$, 
$(aa')_R\otimes b_R=a_Ra'_r\otimes b_{R_r}$, 
$a_R\otimes (bb')_R=a_{R_r}\otimes b_rb'_R$, 
for all $a, a'\in A$ and $b, b'\in B$ (where $r$ and $R$ are two different indices), 
if we define on $A\ot B$ a new multiplication, by 
$(a\ot b)(a'\ot b')=aa'_R\ot b_Rb'$, then this multiplication is associative 
with unit $1\ot 1$. In this case, the map $R$ is called 
a {\bf twisting map} between $A$ and $B$ and the new algebra 
structure on $A\ot B$ is denoted by $A\ot _RB$ and called the 
{\bf twisted tensor product} of $A$ and $B$ afforded by the map $R$. 

We recall from \cite{brz} the construction of  
Brzezi\'{n}ski's crossed product:
\begin{proposition} (\cite{brz}) \label{defbrz}
Let $(A, \mu , 1_A)$ be an (associative unital) algebra and $V$ a 
vector space equipped with a distinguished element $1_V\in V$. Then 
the vector space $A\ot V$ is an associative algebra with unit $1_A\ot 1_V$ 
and whose multiplication has the property that $(a\ot 1_V)(b\ot v)=
ab\ot v$, for all $a, b\in A$ and $v\in V$, if and only if there exist 
linear maps $\sigma :V\ot V\rightarrow A\ot V$ and 
$R:V\ot A\rightarrow A\ot V$  satisfying the following conditions:
\begin{eqnarray}
&&R(1_V\ot a)=a\ot 1_V, \;\;\;R(v\ot 1_A)=1_A\ot v, \;\;\;\forall 
\;a\in A, \;v\in V, \label{brz1} \\
&&\sigma (1_V, v)=\sigma (v, 1_V)=1_A\ot v, \;\;\;\forall 
\;v\in V, \label{brz2} \\
&&R\circ (id_V\ot \mu )=(\mu \ot id_V)\circ (id_A\ot R)\circ (R\ot id_A), 
\label{brz3} \\
&&(\mu \ot id_V)\circ (id_A\ot \sigma )\circ (R\ot id_V)\circ 
(id_V\ot \sigma ) \nonumber \\
&&\;\;\;\;\;\;\;\;\;\;
=(\mu \ot id_V)\circ (id_A\ot \sigma )\circ (\sigma \ot id_V), \label{brz4} \\
&&(\mu \ot id_V)\circ (id_A\ot \sigma )\circ (R\ot id_V)\circ 
(id_V\ot R ) \nonumber \\
&&\;\;\;\;\;\;\;\;\;\;
=(\mu \ot id_V)\circ (id_A\ot R )\circ (\sigma \ot id_A). \label{brz5} 
\end{eqnarray}
If this is the case, the multiplication of $A\ot V$ is given explicitly by
\begin{eqnarray*} 
&&\mu _{A\ot V}=(\mu _2\ot id_V)\circ (id_A\ot id_A\ot \sigma )\circ 
(id_A\ot R\ot id_V),
\end{eqnarray*}
where $\mu _2=\mu \circ (id_A\ot \mu )=\mu \circ (\mu \ot id_A)$. 
We denote by $A\ot _{R, \sigma }V$ this algebra structure and 
call it the {\bf crossed product} afforded by the data $(A, V, R, \sigma )$.  
\end{proposition}

If  $A\ot _{R, \sigma }V$ is a crossed product, we introduce the 
following Sweedler-type notation:
\begin{eqnarray*}
&&R:V\ot A\rightarrow A\ot V, \;\;\;R(v\ot a)=a_R\ot v_R, \\
&&\sigma :V\ot V\rightarrow A\ot V, \;\;\;\sigma (v\ot v')=\sigma _1(v, v') 
\ot \sigma _2(v, v'), 
\end{eqnarray*} 
for all $v, v'\in V$ and $a\in A$. With this notation, the multiplication of 
 $A\ot _{R, \sigma }V$ reads
\begin{eqnarray*}
&&(a\ot v)(a'\ot v')=aa'_R\sigma _1(v_R, v')\ot \sigma _2(v_R, v'), \;\;\;
\forall \;a, a'\in A, \;v, v'\in V.
\end{eqnarray*}

A twisted tensor product is a particular case of a crossed product 
(cf. \cite{guccione}), namely, if $A\ot _RB$ is a twisted tensor product of 
algebras then $A\ot _RB=A\ot _{R, \sigma }B$, where 
$\sigma :B\ot B\rightarrow A\ot B$ 
is given by $\sigma (b\ot b')=1_A\ot bb'$, for all $b, b'\in B$. 
\begin{remark}
The conditions (\ref{brz3}),  (\ref{brz4}) and (\ref{brz5}) for $R$, $\sigma $ 
may be written 
in Sweedler-type notation respectively as 
\begin{eqnarray}
&&(aa')_R\ot v_R=a_Ra'_r\ot v_{R_r}, \label{brz3'} \\
&&\sigma _1(y, z)_R\sigma _1(x_R, \sigma _2(y, z))\ot 
\sigma _2(x_R, \sigma _2(y, z)) \nonumber \\
&&\;\;\;\;\;\;\;\;\;\;=\sigma _1(x, y)\sigma _1(\sigma _2(x, y), z)\ot 
\sigma _2(\sigma _2(x, y), z), \label{brz4'}\\
&&a_{R_r}\sigma _1(v_r, v'_R)\ot \sigma _2(v_r, v'_R)
=\sigma _1(v, v')a_R\ot \sigma _2(v, v')_R, \label{brz5'}
\end{eqnarray}
for all $a, a'\in A$ and $x, y, z, v, v'\in V,$ where $r$ is another copy of $R$. 
\end{remark}
\section{The main result and its particular cases}
\setcounter{equation}{0}
${\;\;\;\;}$
We begin by defining the mirror version of 
Brzezi\'{n}ski's crossed product; the proof is 
similar to the one in \cite{brz} and will be omitted. 
\begin{theorem}\label{mainmirror}
Let $(B, \mu , 1_B)$ be an (associative unital) algebra and 
$W$ a vector space equipped with a distinguished element 
$1_W\in W$. Then the vector space $W\ot B$ is an associative 
algebra with unit $1_W\ot 1_B$ and whose 
multiplication has the property that 
$(w\ot b)(1_W\ot b')=w\ot bb'$, for all $b, b'\in B$ and 
$w\in W$, if and only if there exist linear maps 
$\nu :W\ot W\rightarrow W\ot B$ and 
$P:B\ot W\rightarrow W\ot B$ satisfying the following conditions: 
\begin{eqnarray}
&&P(b\ot 1_W)=1_W\ot b, \;\;\; P(1_B\ot w)=w\ot 1_B, \;\;\;
\forall \; b\in B, \;w\in W, \label{mirtwunit} \\
&&\nu (w\ot 1_W)=\nu (1_W\ot w)=w\ot 1_B, \;\;\;
\forall \;w\in W, \label{mircocunit} \\
&&P\circ (\mu \ot id_W)=(id_W\ot \mu )\circ 
(P\ot id_B)\circ (id_B\ot P), \label{mirtwmap}\\
&&(id_W\ot \mu )\circ (\nu \ot id_B)\circ 
(id_W\ot P)\circ (\nu \ot id_W)\nonumber \\
&&\;\;\;\;\;\;\;\;\;\;
=(id_W\ot \mu )\circ (\nu \ot id_B)\circ (id_W\ot \nu ), 
\label{mir1}\\
&&(id_W\ot \mu )\circ (\nu \ot id_B)\circ (id_W\ot P)
\circ (P\ot id_W)\nonumber \\
&&\;\;\;\;\;\;\;\;\;\;
=(id_W\ot \mu )\circ (P\ot id_B)\circ (id_B\ot \nu ).
\label{mir2}
\end{eqnarray}
If this is the case, the multiplication of $W\ot B$ is given 
explicitely by 
\begin{eqnarray*}
&&\mu _{W\ot B}=(id_W\ot \mu _2)\circ (\nu \ot id_B
\ot id_B)\circ (id_W\ot P\ot id_B), 
\end{eqnarray*}
where $\mu _2=\mu \circ (id_B\ot \mu )=
\mu \circ (\mu \ot id_B)$. We denote by 
$W\overline{\ot}_{P, \nu }B$ this algebra structure 
and call it the {\bf crossed product} 
afforded by the data $(W, B, P, \nu )$. 
\end{theorem}

If $W\oot _{P, \nu }B$ is a crossed product, we use the 
Sweedler-type notation 
\begin{eqnarray*}
&&P:B\ot W\rightarrow W\ot B, \;\;\;P(b\ot w)=w_P\ot b_P, \\
&&\nu :W\ot W\rightarrow W\ot B, \;\;\;\nu (w\ot w')=
\nu _1(w, w')\ot \nu _2(w, w'), 
\end{eqnarray*}
for all $w, w'\in W$ and $b\in B$. With this notation, 
the multiplication of $W\oot _{P, \nu }B$ reads 
\begin{eqnarray*}
&&(w\ot b)(w'\ot b')=\nu _1(w, w'_P)\ot 
\nu _2(w, w'_P)b_Pb', \;\;\;\forall \;b, b'\in B, \;w, w'\in W. 
\end{eqnarray*}

If $A\ot _RB$ is a twisted tensor product product of algebras, 
then $A\ot _RB=A\oot _{R, \nu }B$, where 
$\nu :A\ot A\rightarrow A\ot B$, 
$\nu (a\ot a')=aa'\ot 1_B$ for all $a, a'\in A$. 
\begin{remark}
The conditions (\ref{mirtwmap}), (\ref{mir1}) and (\ref{mir2}) 
may be written down respectively as
\begin{eqnarray}
&&w_P\ot (bb')_P=w_{P_p}\ot b_pb'_P, \label{expmirtwmap}\\
&&\nu _1(\nu _1(w, w'), w''_P)\ot 
\nu _2(\nu _1(w, w'), w''_P)\nu _2(w, w')_P\nonumber \\
&&\;\;\;\;\;\;\;\;\;\;
=\nu _1(w, \nu _1(w', w''))\ot \nu _2(w, \nu _1(w', w''))
\nu _2(w', w''), \label{expmir1} \\
&&\nu _1(w_P, w'_p)\ot \nu _2(w_P, w'_p)b_{P_p}
=\nu _1(w, w')_P\ot b_P\nu _2(w, w'), \label{expmir2}
\end{eqnarray}
for all $w, w', w''\in W$ and $b, b'\in B$, where 
$p$ is another copy of $P$. 
\end{remark}

Now we can prove that under certain circumstances the two 
versions of crossed products may be iterated:
\begin{theorem} \label{mainitercros}
Let $W\oot _{P, \nu }D$ and $D\ot _{R, \sigma }V$ be two 
crossed products. Assume that we have a linear map 
$Q:V\ot W\rightarrow W\ot D\ot V$, with notation 
$Q(v\ot w)=Q_W(v, w)\ot Q_D(v, w)\ot Q_V(v, w)$, 
for all $v\in V$, $w\in W$, such that $Q(1_V\ot w)=
w\ot 1_D\ot 1_V$, $Q(v\ot 1_W)=1_W\ot 1_D\ot v$, 
for all $v\in V$, $w\in W$, and the following conditions are 
satisfied:
\begin{eqnarray}
&&(id_W\ot \mu _D\ot id_V)\circ (id_W\ot id_D\ot R)
\circ (Q\ot id_D)\circ (id_V\ot P)\nonumber \\
&&\;\;\;\;\;\;\;\;\;\;
=(id_W\ot \mu _D\ot id_V)\circ (P\ot id_D\ot id_V)
\circ (id_D\ot Q)\circ (R\ot id_W), \label{braidgen} \\
&&(id_W\ot \mu _D\ot id_V)\circ (id_W\ot id_D\ot R)\circ 
(Q\ot id_D)\circ (id_V\ot \nu )\nonumber \\
&&\;\;\;\;\;\;\;\;\;\;
=(id_W\ot \mu _D\ot id_V)\circ (\nu \ot \mu _D\ot id_V)
\circ (id_W\ot P\ot id_D\ot id_V)\nonumber \\
&&\;\;\;\;\;\;\;\;\;\;\;\;\;\;\;\;\;\;\;
\circ (id_W\ot id_D\ot Q)
\circ (Q\ot id_W), \label{twgen1}\\
&&(id_W\ot \mu _D\ot id_V)\circ (P\ot id_D\ot id_V)\circ 
(id_D\ot Q)\circ (\sigma \ot id_W)\nonumber \\
&&\;\;\;\;\;\;\;\;\;\;
=(id_W\ot \mu _D\ot id_V)\circ (id_W\ot \mu _D\ot \sigma )
\circ (id_W\ot id_D\ot R\ot id_V)\nonumber \\
&&\;\;\;\;\;\;\;\;\;\;\;\;\;\;\;\;\;\;\;
\circ (Q\ot id_D\ot id_V)\circ (id_V\ot Q). \label{twgen2}
\end{eqnarray}
(i) Define the maps 
\begin{eqnarray*}
&&\overline{\sigma}:V\ot V\rightarrow 
(W\oot _{P, \nu }D)\ot V, \\
&&\overline{\sigma }(v\ot v')=(1_W\ot \sigma _1(v, v'))\ot 
\sigma _2(v, v'), \;\;\;\forall \; v, v'\in V, \\
&&\overline{R}:V\ot (W\oot _{P, \nu }D)\rightarrow 
(W\oot _{P, \nu }D)\ot V, \\
&&\overline{R}=(id_W\ot \mu _D\ot id_V)\circ 
(id_W\ot id_D\ot R)\circ (Q\ot id_D)
\end{eqnarray*}
(i.e. for all $v\in V$, $w\in W$, 
$d\in D$ we have 
$\overline{R}(v\ot w\ot d)=Q_W(v, w)\ot 
Q_D(v, w)d_R\ot Q_V(v, w)_R$). 
Then we have a crossed product 
$(W\oot _{P, \nu }D)\ot _{\overline{R}, \overline{\sigma }}V$. \\
(ii) Define the maps 
\begin{eqnarray*}
&&\overline{\nu }:W\ot W\rightarrow 
W\ot (D\ot _{R, \sigma }V), \\
&&\overline{\nu }(w\ot w')=\nu _1(w, w')\ot 
(\nu _2(w, w')\ot 1_V), \;\;\;\forall \; w, w'\in W, \\
&&\overline{P}:(D\ot _{R, \sigma }V)\ot W\rightarrow W\ot 
(D\ot _{R, \sigma }V), \\
&&\overline{P}=(id_W\ot \mu _D\ot id_V)\circ 
(P\ot id_D\ot id_V)\circ (id_D\ot Q)
\end{eqnarray*}
(i.e. for all $v\in V$, $w\in W$, 
$d\in D$ we have $\overline{P}(d\ot v\ot w)=
Q_W(v, w)_P\ot d_PQ_D(v, w)\ot Q_V(v, w)$). 
Then we have a crossed product 
$W\oot _{\overline{P}, \overline{\nu }}(D\ot _{R, \sigma }V)$. \\
(iii) We have an algebra isomorphism 
$(W\oot _{P, \nu }D)\ot _{\overline{R}, \overline{\sigma }}V\cong 
W\oot _{\overline{P}, \overline{\nu }}(D\ot _{R, \sigma }V)$ 
given by the trivial identification.
\end{theorem}
\begin{proof}
Note first that the relations (\ref{braidgen}), 
(\ref{twgen1}) and (\ref{twgen2}) may be written 
down respectively as 
\begin{eqnarray}
&&Q_W(v, w_P)\ot Q_D(v, w_P)d_{P_R}\ot 
Q_V(v, w_P)_R\nonumber \\
&&\;\;\;\;\;
=Q_W(v_R, w)_P\ot d_{R_P}Q_D(v_R, w)\ot 
Q_V(v_R, w), \label{concr1} \\
&&Q_W(v, \nu _1(w, w'))\ot Q_D(v, \nu _1(w, w'))
\nu _2(w, w')_R\ot Q_V(v, \nu _1(w, w'))_R\nonumber \\
&&\;\;\;\;\;
=\nu _1(Q_W(v, w), Q_W(Q_V(v, w), w')_P)\ot 
\nu _2(Q_W(v, w), Q_W(Q_V(v, w), w')_P)\nonumber \\
&&\;\;\;\;\;\;\;\;\;\;
Q_D(v, w)_PQ_D(Q_V(v, w), w')
\ot Q_V(Q_V(v, w), w'), \label{concr2}\\
&&Q_W(\sigma _2(v, v'), w)_P\ot \sigma _1(v, v')_P
Q_D(\sigma _2(v, v'), w)\ot 
Q_V(\sigma _2(v, v'), w)\nonumber \\
&&\;\;\;\;\;
=Q_W(v, Q_W(v', w))\ot Q_D(v, Q_W(v', w))
Q_D(v', w)_R\nonumber \\
&&\;\;\;\;\;\;\;\;\;\;
\sigma _1(Q_V(v, Q_W(v', w))_R, Q_V(v', w))\ot 
\sigma _2(Q_V(v, Q_W(v', w))_R, Q_V(v', w)), \label{concr3}
\end{eqnarray}
for all $d\in D$, $v, v'\in V$ and $w, w'\in W$. 

We will only prove (i) and (iii), while (ii) is similar 
to (i) and left to the reader.\\
\underline{Proof of (i)}:\\[2mm]
The conditions (\ref{brz1}) and (\ref{brz2}) are very easy 
to prove and are left to the reader. We denote as usual 
by $R=r=\mathcal{R}=\overline{R}$ some more copies 
of $R$ and by $p$ another copy of $P$.  \\
\underline{Proof of (\ref{brz3})}:\\[2mm]
Let $v\in V$, $w, w'\in W$ and $d, d'\in D$; we compute:\\[2mm]
${\;\;\;}$
$(\mu _{W\oot _{P, \nu }D}\ot id_V)\circ 
(id_{W\oot _{P, \nu }D}\ot \overline{R})\circ 
(\overline{R}\ot id_{W\oot _{P, \nu }D})
(v\ot w\ot d\ot w'\ot d')$
\begin{eqnarray*}
&=&(\mu _{W\oot _{P, \nu }D}\ot id_V)\circ 
(id_{W\oot _{P, \nu }D}\ot \overline{R})
(Q_W(v, w)\ot Q_D(v, w)d_R\ot Q_V(v, w)_R\ot w'\ot d')\\
&=&(\mu _{W\oot _{P, \nu }D}\ot id_V)
(Q_W(v, w)\ot Q_D(v, w)d_R\ot Q_W(Q_V(v, w)_R, w')\\
&&\ot Q_D(Q_V(v, w)_R, w')d'_r\ot Q_V(Q_V(v, w)_R, w')_r)\\
&=&\nu _1(Q_W(v, w), Q_W(Q_V(v, w)_R, w')_P)\ot 
\nu _2(Q_W(v, w), Q_W(Q_V(v, w)_R, w')_P)\\
&&[Q_D(v, w)d_R]_PQ_D(Q_V(v, w)_R, w')d'_r\ot 
Q_V(Q_V(v, w)_R, w')_r\\
&\overset{(\ref{expmirtwmap})}{=}&
\nu _1(Q_W(v, w), Q_W(Q_V(v, w)_R, w')_{P_p})\ot 
\nu _2(Q_W(v, w), Q_W(Q_V(v, w)_R, w')_{P_p})\\
&&Q_D(v, w)_pd_{R_P}Q_D(Q_V(v, w)_R, w')d'_r\ot 
Q_V(Q_V(v, w)_R, w')_r\\
&\overset{(\ref{concr1})}{=}&
\nu _1(Q_W(v, w), Q_W(Q_V(v, w), w'_P)_p)\ot 
\nu _2(Q_W(v, w), Q_W(Q_V(v, w), w'_P)_p)\\
&&Q_D(v, w)_pQ_D(Q_V(v, w), w'_P)d_{P_R}d'_r\ot 
Q_V(Q_V(v, w), w'_P)_{R_r}\\
&\overset{(\ref{concr2})}{=}&
Q_W(v, \nu _1(w, w'_P))\ot 
Q_D(v, \nu _1(w, w'_P))\nu _2(w, w'_P)_{\mathcal{R}}
d_{P_R}d'_r\ot Q_V(v, \nu _1(w, w'_P))_{\mathcal{R}_{R_r}}\\
&\overset{(\ref{brz3'})}{=}&
Q_W(v, \nu _1(w, w'_P))\ot 
Q_D(v, \nu _1(w, w'_P))[\nu _2(w, w'_P)
d_Pd']_R\ot Q_V(v, \nu _1(w, w'_P))_R\\
&=&\overline{R}(v\ot \nu _1(w, w'_P)\ot \nu _2(w, w'_P)d_Pd')\\
&=&\overline{R}\circ (id_V\ot \mu _{W\oot _{P, \nu }D})
(v\ot w\ot d\ot w'\ot d'), \;\;\;q.e.d.
\end{eqnarray*}
\underline{Proof of (\ref{brz4})}:\\[2mm]
For $v, v', v''\in V$ we compute:\\[2mm]
${\;\;\;}$
$(\mu _{W\oot _{P, \nu }D}\ot id_V)\circ 
(id_{W\oot _{P, \nu }D}\ot \overline{\sigma })
\circ (\overline{R}\ot id_V)\circ 
(id_V\ot \overline{\sigma })(v\ot v'\ot v'')$
\begin{eqnarray*}
&=&(\mu _{W\oot _{P, \nu }D}\ot id_V)\circ 
(id_{W\oot _{P, \nu }D}\ot \overline{\sigma })
\circ (\overline{R}\ot id_V)
(v\ot 1_W\ot \sigma _1(v', v'')\ot \sigma _2(v', v''))\\
&=&(\mu _{W\oot _{P, \nu }D}\ot id_V)\circ 
(id_{W\oot _{P, \nu }D}\ot \overline{\sigma })
(Q_W(v, 1_W)\ot Q_D(v, 1_W)\sigma _1(v', v'')_R\\
&&\ot Q_V(v, 1_W)_R\ot \sigma _2(v', v''))\\
&=&(\mu _{W\oot _{P, \nu }D}\ot id_V)\circ 
(id_{W\oot _{P, \nu }D}\ot \overline{\sigma })
(1_W\ot \sigma _1(v', v'')_R
\ot v_R\ot \sigma _2(v', v''))\\
&=&(\mu _{W\oot _{P, \nu }D}\ot id_V)
(1_W\ot \sigma _1(v', v'')_R\ot 1_W\ot 
\sigma _1(v_R, \sigma _2(v', v''))\ot 
\sigma _2(v_R, \sigma _2(v', v'')))\\
&=&1_W\ot \sigma _1(v', v'')_R
\sigma _1(v_R, \sigma _2(v', v''))\ot 
\sigma _2(v_R, \sigma _2(v', v''))\\
&\overset{(\ref{brz4'})}{=}&
1_W\ot \sigma _1(v, v')
\sigma _1(\sigma _2(v, v'), v'')\ot 
\sigma _2(\sigma _2(v, v'), v'')\\
&=&(\mu _{W\oot _{P, \nu }D}\ot id_V)
(1_W\ot \sigma _1(v, v')\ot 1_W\ot \sigma _1(\sigma _2(v, v'), v'')
\ot \sigma _2(\sigma _2(v, v'), v''))\\
&=&(\mu _{W\oot _{P, \nu }D}\ot id_V)\circ 
(id_{W\oot _{P, \nu }D}\ot \overline{\sigma })
(1_W\ot \sigma _1(v, v')\ot \sigma _2(v, v')\ot v'')\\
&=&(\mu _{W\oot _{P, \nu }D}\ot id_V)\circ 
(id_{W\oot _{P, \nu }D}\ot \overline{\sigma })
\circ (\overline{\sigma }\ot id_V)(v\ot v'\ot v''), \;\;\;q.e.d.
\end{eqnarray*}
\underline{Proof of (\ref{brz5})}:\\[2mm]
For $v, v'\in V$, $w\in W$ and $d\in D$ we compute:\\[2mm]
${\;\;\;}$
$(\mu _{W\oot _{P, \nu }D}\ot id_V)
\circ (id_{W\oot _{P, \nu }D}\ot \overline{\sigma })
\circ (\overline{R}\ot id_V)\circ 
(id_V\ot \overline{R})(v\ot v'\ot w\ot d)$
\begin{eqnarray*}
&=&(\mu _{W\oot _{P, \nu }D}\ot id_V)
\circ (id_{W\oot _{P, \nu }D}\ot \overline{\sigma })
\circ (\overline{R}\ot id_V)(v\ot Q_W(v', w)\\
&&\ot 
Q_D(v', w)d_R\ot Q_V(v', w)_R)\\
&=&(\mu _{W\oot _{P, \nu }D}\ot id_V)
\circ (id_{W\oot _{P, \nu }D}\ot \overline{\sigma })
(Q_W(v, Q_W(v', w))\\
&&\ot Q_D(v, Q_W(v', w))
[Q_D(v', w)d_R]_r\ot Q_V(v, Q_W(v', w))_r\ot 
Q_V(v', w)_R)\\
&=&(\mu _{W\oot _{P, \nu }D}\ot id_V)
(Q_W(v, Q_W(v', w))\ot Q_D(v, Q_W(v', w))
[Q_D(v', w)d_R]_r\ot 1_W\\
&&\ot \sigma _1(Q_V(v, Q_W(v', w))_r, 
Q_V(v', w)_R)\ot \sigma _2(Q_V(v, Q_W(v', w))_r, 
Q_V(v', w)_R))\\
&=&Q_W(v, Q_W(v', w))\ot Q_D(v, Q_W(v', w))
[Q_D(v', w)d_R]_r\\
&&\sigma _1(Q_V(v, Q_W(v', w))_r, 
Q_V(v', w)_R)\ot \sigma _2(Q_V(v, Q_W(v', w))_r, 
Q_V(v', w)_R))\\
&\overset{(\ref{brz3'})}{=}&
Q_W(v, Q_W(v', w))\ot Q_D(v, Q_W(v', w))
Q_D(v', w)_{\mathcal{R}}d_{R_r}\\
&&\sigma _1(Q_V(v, Q_W(v', w))_{\mathcal{R}_r}, 
Q_V(v', w)_R)\ot \sigma _2(Q_V(v, Q_W(v', w))_{\mathcal{R}_r}, 
Q_V(v', w)_R)\\
&\overset{(\ref{brz5'})}{=}&
Q_W(v, Q_W(v', w))\ot Q_D(v, Q_W(v', w))
Q_D(v', w)_{\mathcal{R}}\\
&&\sigma _1(Q_V(v, Q_W(v', w))_{\mathcal{R}}, 
Q_V(v', w))d_R 
\ot \sigma _2(Q_V(v, Q_W(v', w))_{\mathcal{R}}, 
Q_V(v', w))_R\\
&\overset{(\ref{concr3})}{=}&
Q_W(\sigma _2(v, v'), w)_P\ot \sigma _1(v, v')_P
Q_D(\sigma _2(v, v'), w)d_R\ot 
Q_V(\sigma _2(v, v'), w)_R\\
&\overset{(\ref{mircocunit})}{=}&
\nu _1(1_W, Q_W(\sigma _2(v, v'), w)_P)\ot 
\nu _2(1_W, Q_W(\sigma _2(v, v'), w)_P)
\sigma _1(v, v')_P\\
&&Q_D(\sigma _2(v, v'), w)d_R\ot 
Q_V(\sigma _2(v, v'), w)_R\\
&=&(\mu _{W\oot _{P, \nu }D}\ot id_V)(1_W\ot 
\sigma _1(v, v')\ot Q_W(\sigma _2(v, v'), w)\\
&&\ot 
Q_D(\sigma _2(v, v'), w)d_R\ot 
Q_V(\sigma _2(v, v'), w)_R)\\
&=&(\mu _{W\oot _{P, \nu }D}\ot id_V)\circ 
(id_{W\oot _{P, \nu }D}\ot \overline{R})
(1_W\ot \sigma _1(v, v')\ot \sigma _2(v, v')\ot 
w\ot d)\\
&=&(\mu _{W\oot _{P, \nu }D}\ot id_V)\circ 
(id_{W\oot _{P, \nu }D}\ot \overline{R})\circ 
(\overline{\sigma }\ot id_{W\oot _{P, \nu }D})
(v\ot v'\ot w\ot d), \;\;\;q.e.d.
\end{eqnarray*}
\underline{Proof of (iii)}:\\[2mm]
The multiplication in 
$(W\oot _{P, \nu }D)\ot _{\overline{R}, 
\overline{\sigma }}V$ reads:\\[2mm]
${\;\;\;\;\;}$
$((w\ot d)\ot v)((w'\ot d')\ot v')$
\begin{eqnarray*}
&=&(w\ot d)(w'\ot d')_{\overline{R}}
\overline{\sigma }_1(v_{\overline{R}}, v')\ot 
\overline{\sigma }_2(v_{\overline{R}}, v')\\
&=&(w\ot d)(Q_W(v, w')\ot Q_D(v, w')d'_R)
\overline{\sigma }_1(Q_V(v, w')_R, v')\ot 
\overline{\sigma }_2(Q_V(v, w')_R, v')\\
&=&(w\ot d)(Q_W(v, w')\ot Q_D(v, w')d'_R)
(1_W\ot \sigma _1(Q_V(v, w')_R, v'))\ot 
\sigma _2(Q_V(v, w')_R, v')\\
&=&(w\ot d)(Q_W(v, w')\ot Q_D(v, w')d'_R
\sigma _1(Q_V(v, w')_R, v'))\ot 
\sigma _2(Q_V(v, w')_R, v')\\
&=&\nu _1(w, Q_W(v, w')_P)\ot 
\nu _2(w, Q_W(v, w')_P)d_PQ_D(v, w')d'_R
\sigma _1(Q_V(v, w')_R, v')\\
&&\ot \sigma _2(Q_V(v, w')_R, v').
\end{eqnarray*}
The multiplication in 
$W\oot _{\overline{P}, \overline{\nu }}(D\ot _{R, \sigma }V)$
reads:\\[2mm]
${\;\;\;\;\;}$
$(w\ot (d\ot v))(w'\ot (d'\ot v'))$
\begin{eqnarray*}
&=&\overline{\nu }_1(w, w'_{\overline{P}})\ot 
\overline{\nu }_2(w, w'_{\overline{P}})
(d\ot v)_{\overline{P}}(d'\ot v')\\
&=&\overline{\nu }_1(w, Q_W(v, w')_P)\ot 
\overline{\nu }_2(w, Q_W(v, w')_P)
(d_PQ_D(v, w')\ot Q_V(v, w'))(d'\ot v')\\
&=&\nu _1(w, Q_W(v, w')_P)\ot 
(\nu _2(w, Q_W(v, w')_P)\ot 1_V)
(d_PQ_D(v, w')\ot Q_V(v, w'))(d'\ot v')\\
&=&\nu _1(w, Q_W(v, w')_P)\ot 
(\nu _2(w, Q_W(v, w')_P)
d_PQ_D(v, w')\ot Q_V(v, w'))(d'\ot v')\\
&=&\nu _1(w, Q_W(v, w')_P)\ot 
\nu _2(w, Q_W(v, w')_P)
d_PQ_D(v, w')d'_R
\sigma _1(Q_V(v, w')_R, v')\\
&&\ot 
\sigma _2(Q_V(v, w')_R, v'),
\end{eqnarray*}
and we can see that the two multiplications are 
defined by the same formula.
\end{proof}
\begin{definition}
In the hypotheses of Theorem \ref{mainitercros}, the algebra structure 
on $W\ot D\ot V$ arising from (iii) in the theorem will be called the 
{\em iterated crossed product} afforded by the map $Q$. 
\end{definition}

Theorem \ref{mainitercros} admits the following converse: 
\begin{theorem} \label{converse}
Let $W\oot _{P, \nu }D$ and $D\ot _{R, \sigma }V$ be two 
crossed products. Suppose that on $W\ot D\ot V$ we have an 
associative unital algebra structure (with multiplication 
denoted by $\cdot $) such that 
\begin{eqnarray*}
&&W\oot _{P, \nu }D\rightarrow W\ot D\ot V, \;\;\;w\ot d\mapsto 
w\ot d\ot 1_V, \\
&&D\ot _{R, \sigma }V\rightarrow W\ot D\ot V, \;\;\;d\ot v\mapsto 1_W\ot d\ot v,
\end{eqnarray*}
are algebra maps (in particular, it follows that the unit of $W\ot D\ot V$ is 
$1_W\ot 1_D\ot 1_V$) and moreover, for all $w\in W$, $d\in D$, $v\in V$ we have
\begin{eqnarray}
&&w\ot d\ot v=(w\ot 1_D\ot 1_V)\cdot (1_W\ot d\ot 1_V)\cdot 
(1_W\ot 1_D\ot v). \label{canon}
\end{eqnarray}
Then there exists a linear map $Q:V\ot W\rightarrow W\ot D\ot V$ satisfying the 
hypotheses of Theorem \ref{mainitercros} and such that the given algebra 
structure on $W\ot D\ot V$ coincides with the iterated crossed product 
afforded by $Q$. More precisely, the map $Q$ is defined by 
\begin{eqnarray}
&&Q(v\ot w):=(1_W\ot 1_D\ot v)\cdot (w\ot 1_D\ot 1_V), \;\;\;\forall \;\;
v\in V, \;w\in W. \label{formulaQ}
\end{eqnarray}
\end{theorem}
\begin{proof}
We need to prove first that the map $Q$ defined by formula (\ref{formulaQ}) 
satisfies the conditions in Theorem \ref{mainitercros}. Since 
$1_W\ot 1_D\ot 1_V$ is the unit of $W\ot D\ot V$, it is clear that we have 
$Q(1_V\ot w)=w\ot 1_D\ot 1_V$ and $Q(v\ot 1_W)=1_W\ot 1_D\ot v$, 
for all $v\in V$ and $w\in W$. 
We will denote $Q(v\ot w)=(1_W\ot 1_D\ot v)\cdot (w\ot 1_D\ot 1_V)=
Q_W(v, w)\ot Q_D(v, w)\ot Q_V(v, w)$. By regarding $W$, $D$, and $V$ embedded 
canonically into $W\ot D\ot V$, we can regard elements $w$, $d$, $v$ as 
elements in $W\ot D\ot V$; by (\ref{canon}) we can write $w\ot d\ot v=
w\cdot d\cdot v$, for all $w\in W$, $d\in D$, $v\in V$ (in particular, we have 
$Q(v\ot w)=v\cdot w=Q_W(v, w)\cdot Q_D(v, w)\cdot Q_V(v, w)$). 
Because of the algebra embeddings of $W\oot _{P, \nu }D$ 
and $D\ot _{R, \sigma }V$ into $W\ot D\ot V$, we immediately obtain 
\begin{eqnarray}
&&d\cdot w=w_P\cdot d_P, \label{aju1} \\
&&v\cdot d=d_R\cdot v_R, \label{aju2} \\
&&w\cdot w'=\nu _1(w, w')\cdot \nu _2(w, w'), \label{aju3} \\
&&v\cdot v'=\sigma _1(v, v')\cdot \sigma _2(v, v'), \label{aju4}
\end{eqnarray}
for all $v, v'\in V$, $d\in D$, $w, w'\in W$. We prove now that $Q$ satisfies the other 
conditions in Theorem \ref{mainitercros}, namely conditions (\ref{concr1})--(\ref{concr3}).\\
\underline{Proof of (\ref{concr1})}:\\[2mm]
${\;\;\;\;\;}$$Q_W(v, w_P)\ot Q_D(v, w_P)d_{P_R}\ot 
Q_V(v, w_P)_R$
\begin{eqnarray*}
&=&Q_W(v, w_P)\cdot Q_D(v, w_P)\cdot d_{P_R}\cdot 
Q_V(v, w_P)_R\\
&\overset{(\ref{aju2})}{=}&Q_W(v, w_P)\cdot Q_D(v, w_P)\cdot 
Q_V(v, w_P)\cdot d_P\\
&=&Q(v\ot w_P)\cdot d_P\\
&=&v\cdot w_P\cdot d_P\\
&\overset{(\ref{aju1})}{=}&v\cdot d\cdot w, 
\end{eqnarray*}
${\;\;\;\;\;}$$Q_W(v_R, w)_P\ot d_{R_P}Q_D(v_R, w)\ot 
Q_V(v_R, w)$
\begin{eqnarray*}
&=&Q_W(v_R, w)_P\cdot d_{R_P}\cdot Q_D(v_R, w)\cdot 
Q_V(v_R, w)\\
&\overset{(\ref{aju1})}{=}&d_R\cdot Q_W(v_R, w)\cdot 
Q_D(v_R, w)\cdot Q_V(v_R, w)\\
&=&d_R\cdot Q(v_R\ot w)\\
&=&d_R\cdot v_R\cdot w\\
&\overset{(\ref{aju2})}{=}&v\cdot d\cdot w, \;\;\;q.e.d.
\end{eqnarray*}
\underline{Proof of (\ref{concr2})}:\\[2mm]
${\;\;\;\;\;}$$Q_W(v, \nu _1(w, w'))\ot Q_D(v, \nu _1(w, w'))
\nu _2(w, w')_R\ot Q_V(v, \nu _1(w, w'))_R$
\begin{eqnarray*}
&=&Q_W(v, \nu _1(w, w'))\cdot Q_D(v, \nu _1(w, w'))\cdot 
\nu _2(w, w')_R\cdot Q_V(v, \nu _1(w, w'))_R\\
&\overset{(\ref{aju2})}{=}&
Q_W(v, \nu _1(w, w'))\cdot Q_D(v, \nu _1(w, w'))\cdot 
Q_V(v, \nu _1(w, w'))\cdot \nu _2(w, w')\\
&=&Q(v\ot \nu _1(w, w'))\cdot \nu _2(w, w')\\
&=&v\cdot \nu _1(w, w')\cdot \nu _2(w, w')\\
&\overset{(\ref{aju3})}{=}&v\cdot w\cdot w', 
\end{eqnarray*}
\begin{eqnarray*}
&&\nu _1(Q_W(v, w), Q_W(Q_V(v, w), w')_P)\ot 
\nu _2(Q_W(v, w), Q_W(Q_V(v, w), w')_P)Q_D(v, w)_P\nonumber \\
&&\;\;\;\;\;\;\;\;\;\;
Q_D(Q_V(v, w), w')
\ot Q_V(Q_V(v, w), w')\\
&=&\nu _1(Q_W(v, w), Q_W(Q_V(v, w), w')_P)\cdot 
\nu _2(Q_W(v, w), Q_W(Q_V(v, w), w')_P)\nonumber \\
&&\;\;\;\;\;\;\;\;\;\;
\cdot Q_D(v, w)_P
\cdot Q_D(Q_V(v, w), w')
\cdot Q_V(Q_V(v, w), w')\\
&\overset{(\ref{aju3})}{=}&
Q_W(v, w)\cdot Q_W(Q_V(v, w), w')_P
\cdot Q_D(v, w)_P
\cdot Q_D(Q_V(v, w), w')
\cdot Q_V(Q_V(v, w), w')\\
&\overset{(\ref{aju1})}{=}&
Q_W(v, w)\cdot Q_D(v, w)\cdot Q_W(Q_V(v, w), w')
\cdot Q_D(Q_V(v, w), w')
\cdot Q_V(Q_V(v, w), w')\\
&=&Q_W(v, w)\cdot Q_D(v, w)\cdot Q(Q_V(v, w)\ot  w')\\
&=&Q_W(v, w)\cdot Q_D(v, w)\cdot Q_V(v, w)\cdot  w'\\
&=&Q(v\ot w)\cdot w'\\
&=&v\cdot w\cdot w', \;\;\;q.e.d.
\end{eqnarray*}
\underline{Proof of (\ref{concr3})}:\\[2mm]
${\;\;\;\;\;}$
$Q_W(\sigma _2(v, v'), w)_P\ot \sigma _1(v, v')_P
Q_D(\sigma _2(v, v'), w)\ot 
Q_V(\sigma _2(v, v'), w)$
\begin{eqnarray*}
&=&Q_W(\sigma _2(v, v'), w)_P\cdot \sigma _1(v, v')_P\cdot 
Q_D(\sigma _2(v, v'), w)\cdot 
Q_V(\sigma _2(v, v'), w)\\
&\overset{(\ref{aju1})}{=}&
\sigma _1(v, v')\cdot 
Q_W(\sigma _2(v, v'), w)\cdot 
Q_D(\sigma _2(v, v'), w)\cdot 
Q_V(\sigma _2(v, v'), w)\\
&=&\sigma _1(v, v')\cdot 
Q(\sigma _2(v, v')\ot w)\\
&=&\sigma _1(v, v')\cdot 
\sigma _2(v, v')\cdot w\\
&\overset{(\ref{aju4})}{=}&v\cdot v'\cdot w, 
\end{eqnarray*}
\begin{eqnarray*}
&&Q_W(v, Q_W(v', w))\ot Q_D(v, Q_W(v', w))
Q_D(v', w)_R\nonumber \\
&&\;\;\;\;\;\;\;\;\;\;
\sigma _1(Q_V(v, Q_W(v', w))_R, Q_V(v', w))\ot 
\sigma _2(Q_V(v, Q_W(v', w))_R, Q_V(v', w))\\
&=&Q_W(v, Q_W(v', w))\cdot Q_D(v, Q_W(v', w))
\cdot Q_D(v', w)_R\nonumber \\
&&\;\;\;\;\;\;\;\;\;\;
\cdot \sigma _1(Q_V(v, Q_W(v', w))_R, Q_V(v', w))\cdot 
\sigma _2(Q_V(v, Q_W(v', w))_R, Q_V(v', w))\\
&\overset{(\ref{aju4})}{=}&
Q_W(v, Q_W(v', w))\cdot Q_D(v, Q_W(v', w))
\cdot Q_D(v', w)_R\nonumber \\
&&\;\;\;\;\;\;\;\;\;\;
\cdot Q_V(v, Q_W(v', w))_R\cdot Q_V(v', w)\\
&\overset{(\ref{aju2})}{=}&
Q_W(v, Q_W(v', w))\cdot Q_D(v, Q_W(v', w))
\cdot Q_V(v, Q_W(v', w))\cdot Q_D(v', w)\cdot Q_V(v', w)\\
&=&Q(v\ot Q_W(v', w))\cdot Q_D(v', w)\cdot Q_V(v', w)\\
&=&v\cdot Q_W(v', w)\cdot Q_D(v', w)\cdot Q_V(v', w)\\
&=&v\cdot Q(v'\ot w)\\
&=&v\cdot v'\cdot w, \;\;\;q.e.d.
\end{eqnarray*}

Hence, the hypotheses of Theorem \ref{mainitercros} are satisfied, so we can consider 
the iterated crossed product afforded by the map $Q$. The only thing left to prove 
is that the original multiplication of $W\ot D\ot V$ coincides with the multiplication of the 
iterated crossed product (as it appears at the end of the proof of Theorem 
\ref{mainitercros}). So, we express the original multiplication of $W\ot D\ot V$ as 
follows: \\
${\;\;\;\;\;}$
$(w\ot d\ot v)\cdot (w'\ot d'\ot v')$
\begin{eqnarray*}
&=&
w\cdot d\cdot v\cdot w'\cdot d'\cdot v'\\
&=&w\cdot d\cdot Q(v\ot w')\cdot d'\cdot v'\\
&=& w\cdot d\cdot Q_W(v,w')\cdot Q_D(v, w')\cdot Q_V(v, w')\cdot d'\cdot v'\\
&\overset{(\ref{aju1}), (\ref{aju2})}{=}&
w\cdot Q_W(v,w')_P\cdot d_P\cdot Q_D(v, w')\cdot d'_R
\cdot Q_V(v, w')_R\cdot v'\\
&\overset{(\ref{aju3}), (\ref{aju4})}{=}&
\nu _1(w, Q_W(v,w')_P)\cdot \nu _2(w, Q_W(v,w')_P)
d_PQ_D(v, w')d'_R
\sigma _1(Q_V(v, w')_R, v')\\
&&\cdot 
\sigma _2(Q_V(v, w')_R, v')\\
&=&\nu _1(w, Q_W(v, w')_P)\ot 
\nu _2(w, Q_W(v, w')_P)
d_PQ_D(v, w')d'_R
\sigma _1(Q_V(v, w')_R, v')\\
&&\ot 
\sigma _2(Q_V(v, w')_R, v'),
\end{eqnarray*}
and this is exactly the multiplication of the iterated crossed product.
\end{proof}
\begin{example}{\em 
We recall from \cite{jlpvo} what was called there an iterated twisted tensor 
product of algebras. Let $A$, $B$, $C$ be associative unital algebras, 
$R_1:B\ot A\rightarrow A\ot B$, $R_2:C\ot B\rightarrow B\ot C$, 
$R_3:C\ot A\rightarrow A\ot C$ twisting maps  satisfying the braid (or hexagon) 
equation 
\begin{eqnarray*}
&&(id_A\ot R_2)\circ (R_3\ot id_B)\circ (id_C\ot R_1)=
(R_1\ot id_C)\circ (id_B\ot R_3)\circ (R_2\ot id_A). 
\end{eqnarray*}
Then we have an algebra structure on $A\ot B\ot C$ (called the iterated twisted 
tensor product) with unit $1_A\ot 1_B\ot 1_C$ and multiplication 
\begin{eqnarray*}
&&(a\ot b\ot c)(a'\ot b'\ot c')=a(a'_{R_3})_{R_1}\ot b_{R_1}b'_{R_2}\ot 
(c_{R_3})_{R_2}c'.
\end{eqnarray*}
If we consider the crossed products $A\oot _{P, \nu }B=A\ot _{R_1}B$ 
and $B\ot _{R, \sigma }C=B\ot _{R_2}C$, where $P=R_1$, $R=R_2$, 
$\nu (a\ot a')=aa'\ot 1_B$ and $\sigma (c\ot c')=1_B\ot cc'$, then one can easily 
see, by using Theorem \ref{converse}, that the iterated twisted tensor product 
coincides with the iterated crossed product afforded by the map 
$Q:C\ot A\rightarrow A\ot B\ot C$, $Q(c\ot a)=a_{R_3}\ot 1_B\ot c_{R_3}$.}
\end{example}
\begin{example}{\em 
We recall the following construction from \cite{sommer}. Let $H$ be a bialgebra, 
$A$ (respectively $B$) a bialgebra in the category of left (respectively right) 
Yetter-Drinfeld $H$-modules, with module and comodule structures denoted by 
$H\ot A\rightarrow A$, $h\ot a\mapsto h\rightarrow a$, 
$A\rightarrow H\ot A$, $a\mapsto a^1\ot a^2$, 
$B\ot H\rightarrow B$, $b\ot h\mapsto b\leftarrow h$, 
$B\rightarrow B\ot H$, $b\mapsto b^1\ot b^2$,  
and comultiplications denoted by $\Delta (a)=a_1\ot a_2$, $\Delta (b)=b_1\ot b_2$. 
Assume that we are given linear maps $\rightharpoonup :B\ot A\rightarrow A$, 
$\leftharpoonup :B\ot A\rightarrow B$ and $\sharp :B\ot A\rightarrow H$, 
such that $A$ is a left $B$-module via $\rightharpoonup $ and $B$ is a right 
$A$-module via $\leftharpoonup $ and some more conditions (listed in 
\cite{sommer}, pages 39--40) are satisfied. Then by Theorem 3.3 in \cite{sommer}, 
$A\ot H\ot B$ becomes a bialgebra with the comultiplication of a two-sided 
cosmash product, unit $1_A\ot 1_H\ot 1_B$ and the following multiplication: 
\begin{eqnarray*}
&&(a\ot h\ot b)(a'\ot h'\ot b')=a(h_1\rightarrow (b_1^1\rightharpoonup a'_1))
\ot h_2b_1^2(b_2\sharp a'_2)(a'_3)^1h'_1\ot ((b_3\leftharpoonup (a'_3)^2)
\leftarrow h'_2)b'
\end{eqnarray*}
This algebra structure is actually an iterated crossed product. Indeed, consider the 
crossed products $H\ot _{R, \sigma }B=H\# B$ and $A\oot _{P, \nu }H=A\# H$, 
where $H\# B$ and $A\# H$ are respectively the usual right and left 
smash products (so we have $R(b\ot h)=h_1\ot b\leftarrow h_2$, 
$\sigma (b\ot b')=1_H\ot bb'$, $P(h\ot a)=h_1\rightarrow a\ot h_2$, 
$\nu (a\ot a')=aa'\ot 1_H$). By using Theorem \ref{converse}, it turns out that 
Sommerh\"{a}user's algebra structure is the iterated crossed product 
afforded by the map 
\begin{eqnarray*}
&&Q:B\ot A\rightarrow A\ot H\ot B, \;\;\;
Q(b\ot a)=b_1^1\rightharpoonup a_1\ot b_1^2(b_2\sharp a_2)a_3^1\ot 
b_3\leftharpoonup a_3^2.
\end{eqnarray*}
}
\end{example}
\begin{example}{\em 
We recall several things about quasi-Hopf smash products (we use 
terminology and notation as in \cite{bpv}, \cite{bpvo}). Let $H$ be a 
quasi-bialgebra  and $A$ (respectively $B$) 
a left (respectively right) $H$-module algebra. We can consider the left 
(respectively right) smash product $A\# H$ (respectively $H\# B$), 
which is an associative algebra having $A\ot H$ (respectively 
$H\ot B$) as underlying vector space, multiplication 
$(a\# h)(a'\# h')=(x^1\cdot a)(x^2h_1\cdot a')\# x^3h_2h'$ 
(respectively $(h\# b)(h'\# b')=hh'_1x^1\# (b\cdot h'_2x^2)(b'\cdot x^3)$) 
and unit $1_A\# 1_H$ (respectively $1_H\# 1_B$) (we denoted as 
usual $\Phi =X^1\ot X^2\ot X^3$ and $\Phi ^{-1}=x^1\ot x^2\ot x^3$ 
the associator of $H$ and its inverse).  We can consider also the 
so-called two-sided smash product $A\# H\# B$, which is an associative 
algebra structure on $A\ot H\ot B$ with unit $1_A\ot 1_H\ot 1_B$ and 
multiplication 
\begin{eqnarray*}
(a\# h\# b)(a'\# h'\# b')=
(x^1\cdot a)(x^2h_1y^1\cdot a')\# x^3h_2y^2h'_1z^1\# 
(b\cdot y^3h'_2z^2)(b'\cdot z^3), 
\end{eqnarray*}
 where 
$\Phi ^{-1}=y^1\ot y^2\ot y^3=z^1\ot z^2\ot z^3$ are two more copies of 
$\Phi ^{-1}$. 

One can see that $A\# H$ and $H\# B$ are crossed products, namely 
$A\# H=A\oot _{P, \nu }H$ and $H\# B=H\ot _{R, \sigma }B$, where
\begin{eqnarray*}
&&P:H\ot A\rightarrow A\ot H, \;\;\;P(h\ot a)=h_1\cdot a\ot h_2, \\
&&\nu :A\ot A\rightarrow A\ot H, \;\;\;\nu (a\ot a')=
(x^1\cdot a)(x^2\cdot a')\ot x^3,\\
&&R:B\ot H\rightarrow H\ot B, \;\;\;R(b\ot h)=h_1\ot b\cdot h_2, \\
&&\sigma :B\ot B\rightarrow H\ot B, \;\;\;\sigma (b\ot b')=
x^1\ot (b\cdot x^2)(b'\cdot x^3). 
\end{eqnarray*}
An easy computation shows that for the two-sided smash product 
$A\# H\# B$ we have $(a\# h\# b)=(a\# 1_H\# 1_B)
(1_A\# h\# 1_B)(1_A\# 1_H\# b)$, for all $a\in A$, $h\in H$ and $b\in B$, and 
that $A\# H\# B$ is the iterated crossed product obtained from 
$A\# H$ and $H\# B$ via the map $Q:B\ot A\rightarrow A\ot H\ot B$, 
$Q(b\ot a)=x^1\cdot a\ot x^2\ot b\cdot x^3$.
}
\end{example}



\begin{thebibliography}{99}
\bibitem{brz}
T. Brzezi\'{n}ski, Crossed products by a coalgebra, {\sl Comm. Algebra} 
{\bf 25} (1997), 3551--3575. 

\bibitem{bpv}
D. Bulacu, F. Panaite, F. Van Oystaeyen, 
Quasi-Hopf algebra actions and smash products, 
{\sl Comm. Algebra} {\bf 28} (2000), 631--651. 

\bibitem{bpvo}
D. Bulacu, F. Panaite, F. Van Oystaeyen, Generalized diagonal crossed  
products and smash products for quasi-Hopf algebras. Applications, 
{\sl Comm. Math. Phys.} {\bf 266} (2006), 355--399.

\bibitem{Cap}
A. Cap, H. Schichl, J. Vanzura, On twisted tensor products of algebras, 
{\sl Comm. Algebra} {\bf 23} (1995), 4701--4735.

\bibitem{guccione}
C. Di Luigi, J. A. Guccione, J. J. Guccione, 
Brzezi\'{n}ski's crossed products and braided Hopf crossed products, 
{\sl Comm. Algebra} {\bf 32} (2004), 3563--3580.

\bibitem{jlpvo}
P. Jara Mart\'{i}nez, J. L\'{o}pez Pe\~{n}a, F. Panaite, F. Van Oystaeyen,
On iterated twisted tensor products of algebras,  
{\sl Internat. J. Math.} {\bf 19} (2008), 1053--1101.

\bibitem{panaite}
F. Panaite, Invariance under twisting for crossed products, 
{\sl Proc. Amer. Math. Soc.} {\bf 140} (2012), 755--763. 

\bibitem{sommer}
Y. Sommerh\"{a}user, Deformed enveloping algebras, {\sl New York J. Math.} 
{\bf 2} (1996), 35--58

\bibitem{VanDaele}
A. Van Daele, S. Van Keer, The Yang--Baxter and Pentagon equation, 
{\sl Compositio Math.} {\bf 91} (1994), 201--221.


\end{thebibliography}
\end{document}